\documentclass{amsart}
\usepackage{amsfonts}

\newtheorem{theorem}{Theorem}
\theoremstyle{plain}

\newtheorem{corollary}{Corollary}

\newtheorem{definition}{Definition}

\newtheorem{lemma}{Lemma}

\newtheorem{remark}{Remark}

\numberwithin{equation}{section}
\input{tcilatex}

\begin{document}
\title{Inequalities for a number of different types of convexity}
\author{Merve Avci Ardic}
\address{Adiyaman University, Faculty of Science and Arts, Department of
Mathematics, 02040, Adiyaman}
\email{mavci@adiyaman.edu.tr}

\begin{abstract}
Some inequalities for different types of convexity are established.
\end{abstract}

\maketitle

\section{introduction}

In this section, some definitions of different types of convexity will be
reminded.

In \cite{1}, Hudzik and Maligranda considered among others the class of
functions which are s-convex in the second sense.

\begin{definition}
\label{def 1.1} A function $f:%
\mathbb{R}
^{+}\rightarrow 
\mathbb{R}
,$ where $%
\mathbb{R}
^{+}=[0,\infty ),$ is said to be $s-$convex in the second sense if%
\begin{equation*}
f(\alpha u+\beta v)\leq \alpha ^{s}f(u)+\beta ^{s}f(v)
\end{equation*}%
for all $u,v\in \lbrack 0,\infty ),$ $\alpha ,\beta \geq 0$ with $\alpha
+\beta =1$ and $s$ fixed in $(0,1].$ This class of s-convex functions in the
second sense is usually denoted by $K_{s}^{2}.$
\end{definition}

\ $s-$convexity reduces the ordinary convexity of functions defined on $%
[0,\infty )$ for $s=1.$

For some information about convexity and $s-$convexity it is possible to
refer to \cite{1}-\cite{7}.

\begin{definition}
\label{def 1.2} \cite{8} A function $f:I\rightarrow 
\mathbb{R}
$ is said to be $quasi-$convex if for all $x,y\in I$ and all $\alpha \in
\lbrack 0,1],f(\alpha x+(1-\alpha )y)\leq \max (f(x),f(y)).$
\end{definition}

Godunova and Levin introduced the following concept in the paper \cite{9}.

\begin{definition}
\label{def 1.3} A function $f:I\rightarrow 
\mathbb{R}
$ is said to be belong to the class $Q(I)$ if it is nonnegative and for all $%
x,y\in I$ and $\lambda \in (0,1),$ satisfies the inequality%
\begin{equation*}
f(\lambda x+(1-\lambda )y)\leq \frac{f(x)}{\lambda }+\frac{f(y)}{1-\lambda }.
\end{equation*}
\end{definition}

\begin{definition}
\label{def 1.4} \cite{10} A function $f:I\rightarrow 
\mathbb{R}
$ is said to be belong to the class $P(I)$ if it is nonnegative and for all $%
x,y\in I$ and $\lambda \in \lbrack 0,1],$ satisfies the following inequality%
\begin{equation*}
f(\lambda x+(1-\lambda )y)\leq f(x)+f(y).
\end{equation*}
\end{definition}

For some results, generalizations and applications about quasi-convex
functions, $P-$convex functions and $Q(I)$ class functions see \cite{8}-\cite%
{17}.\ \ \ \ \ 

The main aim of this paper is establish some integral inequalities for
different types of convexity via Lemma \ref{lem 2.1}.\ \ \ \ 

\section{\ main results}

The results are obtained via following lemma.

\begin{lemma}
\label{lem 2.1} Let $f:[a,b]\rightarrow 
\mathbb{R}
$ be a continuous function on $[a,b]$ such that $f\in L[a,b]$ with $a<b.$
For some fixed $p,q>0,$ following equality holds: 
\begin{eqnarray*}
&&\int_{a}^{b}\left( x-a\right) ^{p}\left( b-x\right) ^{q}f(x)f(a+b-x)dx \\
&=&\left( b-a\right) ^{p+q+1}\int_{0}^{1}\left( 1-t\right)
^{p}t^{q}f(ta+\left( 1-t\right) b)f\left( \left( 1-t\right) a+tb\right) dt.
\end{eqnarray*}
\end{lemma}

First result is about $s-$convex functions.

\begin{proof}
Changing the variable $x=ta+(1-t)b$ and simple calculations proceed the
required result.
\end{proof}

\begin{theorem}
\label{teo 2.1} Let $f:[a,b]\subset \lbrack 0,\infty )\rightarrow 
\mathbb{R}
^{+}$ be a continuous function on $[a,b]$ such that $f\in L[a,b]$ with 0$%
\leq a<b<\infty .$ If $f$ is $s-$convex in the second sense, for some fixed $%
p,q>0$ and $s\in (0,1]$ the following inequality holds 
\begin{eqnarray*}
&&\int_{a}^{b}\left( x-a\right) ^{p}\left( b-x\right) ^{q}f(x)f(a+b-x)dx \\
&\leq &\frac{\left( b-a\right) ^{p+q+1}}{2}\left\{ \left(
f^{2}(a)+f^{2}(b)\right) \left[ \beta (p+1,2s+q+1)+\beta (q+1,2s+p+1)\right]
\right. \\
&&+\left. 4f(a)f(b)\beta (p+s+1,q+s+1)\right\}
\end{eqnarray*}%
where $\beta \left( m,n\right) =\int_{0}^{1}t^{m-1}\left( 1-t\right)
^{n-1}dt,$ $m,n>0$ is the Euler Beta function.
\end{theorem}

\begin{proof}
Using the inequality $cd\leq \frac{1}{2}[c^{2}+d^{2}]$ $c,d\in 
\mathbb{R}
^{+}$ in the right hand side in Lemma \ref{lem 2.1}, we obtain%
\begin{eqnarray*}
&&\int_{a}^{b}\left( x-a\right) ^{p}\left( b-x\right) ^{q}f(x)f(a+b-x)dx \\
&\leq &\frac{\left( b-a\right) ^{p+q+1}}{2}\int_{0}^{1}\left( 1-t\right)
^{p}t^{q}\left\{ \left[ f(ta+\left( 1-t\right) b)\right] ^{2}+\left[ f\left(
\left( 1-t\right) a+tb\right) \right] ^{2}\right\} dt.
\end{eqnarray*}%
Since $f$ is $s-$convex in the second sense, we can write%
\begin{eqnarray*}
&&\int_{a}^{b}\left( x-a\right) ^{p}\left( b-x\right) ^{q}f(x)f(a+b-x)dx \\
&\leq &\frac{\left( b-a\right) ^{p+q+1}}{2}\int_{0}^{1}\left( 1-t\right)
^{p}t^{q}\left\{ \left[ t^{s}f(a)+(1-t)^{s}f(b)\right] ^{2}+\left[
(1-t)^{s}f(a)+t^{s}f(b)\right] ^{2}\right\} dt \\
&=&\frac{\left( b-a\right) ^{p+q+1}}{2}\int_{0}^{1}\left( 1-t\right)
^{p}t^{q}\left\{ \left( f^{2}(a)+f^{2}(b)\right) \left( t^{2s}+\left(
1-t\right) ^{2s}\right) +4t^{s}(1-t)^{s}f(a)f(b)\right\} \\
&=&\frac{\left( b-a\right) ^{p+q+1}}{2}\left\{ \left(
f^{2}(a)+f^{2}(b)\right) \left[ \int_{0}^{1}\left( 1-t\right)
^{p}t^{q+2s}dt+\int_{0}^{1}\left( 1-t\right) ^{p+2s}t^{q}dt\right] \right. \\
&&+\left. 4f(a)f(b)\int_{0}^{1}t^{q+s}(1-t)^{p+s}dt\right\} .
\end{eqnarray*}%
If we use the following equalities above we get the required result:%
\begin{equation*}
\int_{0}^{1}\left( 1-t\right) ^{p}t^{q+2s}dt=\beta (p+1,2s+q+1),
\end{equation*}%
\begin{equation*}
\int_{0}^{1}\left( 1-t\right) ^{p+2s}t^{q}dt=\beta (q+1,2s+p+1)
\end{equation*}%
and%
\begin{equation*}
\int_{0}^{1}t^{q+s}(1-t)^{p+s}dt=\beta (p+s+1,q+s+1).
\end{equation*}
\end{proof}

\begin{corollary}
\label{co 2.1} In Theorem \ref{teo 2.1}, if $p=q$ following inequality holds:%
\begin{eqnarray*}
&&\int_{a}^{b}\left( x-a\right) ^{p}\left( b-x\right) ^{p}f(x)f(a+b-x)dx \\
&\leq &\frac{\left( b-a\right) ^{2p+1}}{2}\left\{ \left(
f^{2}(a)+f^{2}(b)\right) \left[ \beta (p+1,2s+p+1)+\beta (p+1,2s+p+1)\right]
\right. \\
&&+\left. 4f(a)f(b)\beta (p+s+1,p+s+1)\right\} .
\end{eqnarray*}
\end{corollary}

\begin{corollary}
\label{co 2.2} In Theorem \ref{teo 2.1}, if $f$ is symmetric function, $%
f(x)=f(a+b-x),$ following inequality holds:%
\begin{eqnarray*}
&&\int_{a}^{b}\left( x-a\right) ^{p}\left( b-x\right) ^{q}f^{2}(x)dx \\
&\leq &\frac{\left( b-a\right) ^{p+q+1}}{2}\left\{ \left(
f^{2}(a)+f^{2}(b)\right) \left[ \beta (p+1,2s+q+1)+\beta (q+1,2s+p+1)\right]
\right. \\
&&+\left. 4f(a)f(b)\beta (p+s+1,q+s+1)\right\} .
\end{eqnarray*}
\end{corollary}

\begin{corollary}
\label{co 2.3} In Theorem \ref{teo 2.1}, if we choose $s=1,$ following
inequality holds 
\begin{eqnarray*}
&&\int_{a}^{b}\left( x-a\right) ^{p}\left( b-x\right) ^{q}f(x)f(a+b-x)dx \\
&\leq &\frac{\left( b-a\right) ^{p+q+1}}{2}\left\{ \left(
f^{2}(a)+f^{2}(b)\right) \left[ \beta (p+1,q+3)+\beta (q+1,p+3)\right]
\right. \\
&&+\left. 4f(a)f(b)\beta (p+2,q+2)\right\}
\end{eqnarray*}%
for convex functions.
\end{corollary}

\begin{corollary}
\label{co 2.4} In Corollary \ref{co 2.3}, if we choose $p=q$ and $%
f(x)=f(a+b-x),$ we obtain the following inequalities respectively%
\begin{eqnarray*}
&&\int_{a}^{b}\left( x-a\right) ^{p}\left( b-x\right) ^{p}f(x)f(a+b-x)dx \\
&\leq &\left( b-a\right) ^{2p+1}\left\{ \left( f^{2}(a)+f^{2}(b)\right)
\beta (p+1,p+3)\right. \\
&&+\left. 2f(a)f(b)\beta (p+2,p+2)\right\}
\end{eqnarray*}%
and%
\begin{eqnarray*}
&&\int_{a}^{b}\left( x-a\right) ^{p}\left( b-x\right) ^{q}f^{2}(x)dx \\
&\leq &\frac{\left( b-a\right) ^{p+q+1}}{2}\left\{ \left(
f^{2}(a)+f^{2}(b)\right) \left[ \beta (p+1,q+3)+\beta (q+1,p+3)\right]
\right. \\
&&+\left. 4f(a)f(b)\beta (p+2,q+2)\right\}
\end{eqnarray*}%
for convex functions.
\end{corollary}

Following results are about $quasi-$convex functions.

\begin{theorem}
\label{teo 2.2} Let $f:[a,b]\subset \lbrack 0,\infty )\rightarrow 
\mathbb{R}
^{+}$ be a continuous function on $[a,b]$ such that $f\in L[a,b]$ with 0$%
\leq a<b<\infty .$ If $f$ is $quasi-$convex , for some fixed $p,q>0$ the
following inequality holds 
\begin{eqnarray*}
&&\int_{a}^{b}\left( x-a\right) ^{p}\left( b-x\right) ^{q}f(x)f(a+b-x)dx \\
&\leq &\left( b-a\right) ^{p+q+1}\left( \max \left\{ f(a),f(b)\right\}
\right) ^{2}\beta (p+1,q+1)
\end{eqnarray*}%
where $\beta $ is the Euler Beta function.
\end{theorem}

\begin{proof}
Using the inequality $cd\leq \frac{1}{2}[c^{2}+d^{2}]$ $c,d\in 
\mathbb{R}
^{+}$ in the right hand side in Lemma \ref{lem 2.1} and quasi convexity of $%
f,$ we obtain 
\begin{eqnarray*}
&&\int_{a}^{b}\left( x-a\right) ^{p}\left( b-x\right) ^{q}f(x)f(a+b-x)dx \\
&\leq &\frac{\left( b-a\right) ^{p+q+1}}{2}\int_{0}^{1}\left( 1-t\right)
^{p}t^{q}\left\{ \left[ f(ta+\left( 1-t\right) b)\right] ^{2}+\left[ f\left(
\left( 1-t\right) a+tb\right) \right] ^{2}\right\} dt. \\
&\leq &\left( b-a\right) ^{p+q+1}\left( \max \left\{ f(a),f(b)\right\}
\right) ^{2}\int_{0}^{1}\left( 1-t\right) ^{p}t^{q}dt.
\end{eqnarray*}%
The proof is completed.
\end{proof}

\begin{corollary}
\label{co 2.5} In Theorem \ref{teo 2.2},
\end{corollary}

\begin{itemize}
\item If $f$ is increasing, the following inequality holds%
\begin{eqnarray*}
&&\int_{a}^{b}\left( x-a\right) ^{p}\left( b-x\right) ^{q}f(x)f(a+b-x)dx \\
&\leq &\left( b-a\right) ^{p+q+1}f^{2}(b)\beta (p+1,q+1).
\end{eqnarray*}

\item If $f$ is decreasing, the following inequality holds%
\begin{eqnarray*}
&&\int_{a}^{b}\left( x-a\right) ^{p}\left( b-x\right) ^{q}f(x)f(a+b-x)dx \\
&\leq &\left( b-a\right) ^{p+q+1}f^{2}(a)\beta (p+1,q+1).
\end{eqnarray*}
\end{itemize}

Following \ result is about $P-$convexity.

\begin{theorem}
\label{teo 2.3} Let $f:[a,b]\subset \lbrack 0,\infty )\rightarrow 
\mathbb{R}
$ be a continuous function on $[a,b]$ such that $f\in L[a,b]$ with $0\leq
a<b<\infty .$ If $f$ is $P-$convex , for some fixed $p,q>0$ the following
inequality holds 
\begin{eqnarray*}
&&\int_{a}^{b}\left( x-a\right) ^{p}\left( b-x\right) ^{q}f(x)f(a+b-x)dx \\
&\leq &\left( b-a\right) ^{p+q+1}\left\{ f(a)+f(b)\right\} ^{2}\beta
(p+1,q+1)
\end{eqnarray*}%
where $\beta $ is the Euler Beta function.
\end{theorem}

\begin{proof}
Using the inequality $cd\leq \frac{1}{2}[c^{2}+d^{2}]$ $c,d\in 
\mathbb{R}
^{+}$ in the right hand side in Lemma \ref{lem 2.1}, $P-$convexity of $f$
and the definition of $\beta $ function, we get the desired result.
\end{proof}

The last theorem is for $Q(I)$ class functions.

\begin{theorem}
\label{teo 2.4} Let $f:[a,b]\subset \lbrack 0,\infty )\rightarrow 
\mathbb{R}
$ be a continuous function on $[a,b]$ such that $f\in L[a,b]$ with $0\leq
a<b<\infty .$ If $f$ belongs to $Q(I)$ class, the following inequality holds%
\begin{eqnarray*}
&&\int_{a}^{b}\left( x-a\right) ^{p}\left( b-x\right) ^{q}f(x)f(a+b-x)dx \\
&\leq &\frac{\left( b-a\right) ^{p+q+1}}{2}\left\{ \left[ f^{2}(a)+f^{2}(b)%
\right] \left( \beta (p+1,q-1)+\beta (p-1,q+1)\right) \right. \\
&&+\left. 4f(a)f(b)\beta (p,q)\right\}
\end{eqnarray*}%
for some fixed $p,q>1$ and $t\in (0,1).$
\end{theorem}

\begin{proof}
Using the inequality $cd\leq \frac{1}{2}[c^{2}+d^{2}]$ $c,d\in 
\mathbb{R}
^{+}$ in the right hand side in Lemma \ref{lem 2.1}$,$ we obtain 
\begin{eqnarray*}
&&\int_{a}^{b}\left( x-a\right) ^{p}\left( b-x\right) ^{q}f(x)f(a+b-x)dx \\
&\leq &\frac{\left( b-a\right) ^{p+q+1}}{2}\int_{0}^{1}\left( 1-t\right)
^{p}t^{q}\left\{ \left[ f(ta+\left( 1-t\right) b)\right] ^{2}+\left[ f\left(
\left( 1-t\right) a+tb\right) \right] ^{2}\right\} dt.
\end{eqnarray*}%
Since $f$ belongs to $Q(I),$we can write%
\begin{eqnarray*}
&&\int_{a}^{b}\left( x-a\right) ^{p}\left( b-x\right) ^{q}f(x)f(a+b-x)dx \\
&\leq &\frac{\left( b-a\right) ^{p+q+1}}{2}\int_{0}^{1}\left( 1-t\right)
^{p}t^{q}\left\{ \left[ \frac{f(a)}{t}+\frac{f(b)}{1-t}\right] ^{2}+\left[ 
\frac{f(a)}{1-t}+\frac{f(b)}{t}\right] ^{2}\right\} dt \\
&=&\frac{\left( b-a\right) ^{p+q+1}}{2}\left\{ \left[ f^{2}(a)+f^{2}(b)%
\right] \left( \int_{0}^{1}\left( 1-t\right)
^{p}t^{q-2}dt+\int_{0}^{1}\left( 1-t\right) ^{p-2}t^{q}dt\right) \right. \\
&&+\left. 4f(a)f(b)\int_{0}^{1}\left( 1-t\right) ^{p-1}t^{q-1}dt\right\} .
\end{eqnarray*}%
If we use the following equalities above we get the required result:%
\begin{equation*}
\int_{0}^{1}\left( 1-t\right) ^{p}t^{q-2}dt=\beta (p+1,q-1),
\end{equation*}%
\begin{equation*}
\int_{0}^{1}\left( 1-t\right) ^{p-2}t^{q}dt=\beta (p-1,q+1)
\end{equation*}%
and%
\begin{equation*}
\int_{0}^{1}\left( 1-t\right) ^{p-1}t^{q-1}dt=\beta (p,q).
\end{equation*}
\end{proof}

\begin{corollary}
\label{co 2.6} In Theorem \ref{teo 2.4}, if $f(a)=f(b)$ the following
inequality holds 
\begin{eqnarray*}
&&\int_{a}^{b}\left( x-a\right) ^{p}\left( b-x\right) ^{q}f(x)f(a+b-x)dx \\
&\leq &\left( b-a\right) ^{p+q+1}f^{2}(a)\left\{ \left( \beta
(p+1,q-1)+2\beta (p,q)+\beta (p-1,q+1)\right) \right\}
\end{eqnarray*}%
where $p,q>1$ and $\beta $ is Euler beta function.
\end{corollary}

\begin{remark}
\label{rem 2.1} One may get some results for other types of convexity via
Lemma \ref{lem 2.1}.
\end{remark}

\end{document}